\documentclass[11pt]{article}
\usepackage{amsfonts}
\usepackage{latexsym}
\usepackage{amsmath}
\usepackage{amssymb}
\usepackage{amsthm}
\usepackage{amsmath}
\usepackage{verbatim}
\usepackage{bibspacing}

\newtheorem{thm}{Theorem}[section]
\newtheorem*{thm*}{Theorem}

\newtheorem{lem}[thm]{Lemma}

\newtheorem*{prop*}{Proposition}

\newtheorem*{conj*}{Conjecture}

\newtheorem*{dfn*}{Definition}
\theoremstyle{definition}
\newtheorem{rem}[thm]{\textbf{Remark}}
\newtheorem*{rmk*}{Remark}
\newtheorem*{fact*}{Fact}

\theoremstyle{proof}

\numberwithin{equation}{section}

\newcommand{\vol}{\textrm{Vol}}
\newcommand{\norm}[1]{\left\Vert#1\right\Vert}

\newcommand{\abs}[1]{\left\vert#1\right\vert}
\newcommand{\set}[1]{\left\{#1\right\}}
\newcommand{\brac}[1]{\left(#1\right)}
\newcommand{\scalar}[1]{\left \langle #1 \right \rangle}
\newcommand{\sscalar}[1]{\langle #1 \rangle}
\newcommand{\Real}{\mathbb{R}}

\newcommand{\eps}{\epsilon}

\newcommand{\Ric}{\mbox{\rm{Ric}}}
\newcommand{\CD}{\text{CD}}
\newcommand{\II}{\text{II}}

\renewcommand{\div}{\text{div}}
\renewcommand{\S}{\mathcal{S}}
\newcommand{\Var}{\text{\rm{Var}}}

\def\Xint#1{\mathchoice
{\XXint\displaystyle\textstyle{#1}}{\XXint\textstyle\scriptstyle{#1}}{\XXint\scriptstyle\scriptscriptstyle{#1}}{\XXint\scriptscriptstyle\scriptscriptstyle{#1}}\!\int}
\def\XXint#1#2#3{{\setbox0=\hbox{$#1{#2#3}{\int}$}
\vcenter{\hbox{$#2#3$}}\kern-.5\wd0}}

\def\dashint{\Xint-}

\begin{document}

\date{}

\title{Brascamp--Lieb type inequalities on weighted Riemannian manifolds with boundary}
\author{Alexander V. Kolesnikov\textsuperscript{1} and Emanuel Milman\textsuperscript{2}}

\footnotetext[1]{Faculty of Mathematics, Higher School of Economics, Moscow, Russia. 
Supported by RFBR project 14-01-00237 and the DFG project CRC 701.
This study (research grant No 14-01-0056) was supported by The National Research University -- Higher School of Economics' Academic Fund Program in 2014/2015.
Emails: akolesnikov@hse.ru, sascha77@mail.ru.}

\footnotetext[2]{Department of Mathematics, Technion - Israel
Institute of Technology, Haifa 32000, Israel. Supported by ISF (grant no. 900/10), BSF (grant no. 2010288) and Marie-Curie Actions (grant no. PCIG10-GA-2011-304066).
The research leading to these results is part of a project that has received funding from the European Research Council (ERC) under the European Union's Horizon 2020 research and innovation programme (grant agreement No 637851).
Email: emilman@tx.technion.ac.il.}

\begingroup    \renewcommand{\thefootnote}{}    \footnotetext{2010 Mathematics Subject Classification: 53C21,58J32,58J50.}
\endgroup

\maketitle

\begin{abstract}
It is known that by dualizing the Bochner--Lichnerowicz--Weitzenb\"{o}ck formula, one obtains Poincar\'e-type inequalities on Riemannian manifolds equipped with a density, which satisfy the Bakry--\'Emery Curvature-Dimension condition (combining a lower bound on its generalized Ricci curvature and an upper bound on its generalized dimension).
When the manifold has a boundary, an appropriate generalization of the Reilly formula may be used instead. By systematically dualizing this formula for various combinations of boundary conditions of the domain (convex, mean-convex) and the function (Neumann, Dirichlet), we obtain new Brascamp--Lieb type inequalities on the manifold. 
All previously known inequalities of Lichnerowicz, Brascamp--Lieb, Bobkov--Ledoux and Veysseire are recovered, extended to the Riemannian setting and generalized into a single unified formulation, and their appropriate versions in the presence of a boundary are obtained. Our framework allows to encompass the entire class of Borell's convex measures, including heavy-tailed measures, and extends the latter class to weighted-manifolds having negative generalized dimension.
\end{abstract}

\section{Introduction}

Throughout the paper we consider a compact \emph{weighted-manifold} $(M,g,\mu)$, namely a compact smooth connected and oriented $n$-dimensional Riemannian manifold $(M,g)$ with boundary $\partial M$, equipped with a measure:
\[
\mu = \exp(-V) d {\vol}_M ~,
\] 
where $\vol_M$ is the Riemannian volume form on $M$ and $V \in C^2(M)$ is twice continuously differentiable. The boundary $\partial M$ is assumed to be a $C^2$ manifold with outer unit-normal $\nu = \nu_{\partial M}$. The corresponding symmetric diffusion operator with invariant measure $\mu$, which is called the weighted-Laplacian, is given by:
\[
L  = L_{(M,g,\mu)} := \exp(V) \div( \exp(-V) \nabla) = \Delta - \scalar{\nabla V,\nabla} ~,
\]
where $\scalar{\cdot,\cdot}$ denotes the Riemannian metric $g$,  $\nabla = \nabla_g$ denotes the Levi-Civita connection, $\div = \div_g =  tr(\nabla \cdot)$ denotes the Riemannian divergence operator, and $\Delta = \div \nabla$ is the Laplace-Beltrami operator.  Indeed, note that with these generalized notions, the usual integration by parts formula is satisfied for $f,g \in C^2(M)$: \[
\int_M L(f) g d\mu = \int_{\partial M} f_\nu g d\mu_{\partial M} - \int_M \scalar{\nabla f,\nabla g} d\mu = \int_{\partial M} (f_\nu g - g_\nu f) d\mu_{\partial M} +  \int_M L(g) f d\mu ~,
\]
where  $u_\nu = \nu \cdot u$ and $\mu_{\partial M} := \exp(-V) d\vol_{\partial M}$. 

The second fundamental form $\II = \II_{\partial M}$ of $\partial M \subset M$ at $x \in \partial M$ is as usual (up to sign) defined by $\II_x(X,Y) = \scalar{\nabla_X \nu, Y}$, $X,Y \in T \partial M$. The quantities
\[
H_g(x) := tr(\II_x) ~,~ H_\mu (x) := H_g(x) - \scalar{\nabla V(x) ,\nu(x)} ~,
\]
are called the Riemannian mean-curvature and \emph{generalized} mean-curvature of $\partial M$ at $x \in \partial M$, respectively. It is well-known that $H_g$ governs the first variation of $\vol_{\partial M}$ under the normal-map $t \mapsto \exp(t \nu)$, and similarly $H_\mu$ governs the first variation of $\exp(-V) d\vol_{\partial M}$ in the weighted-manifold setting, see e.g. \cite{EMilmanGeometricApproachPartI}. 
\medskip

In the purely Riemannian setting, it is classical that positive lower bounds on the Ricci curvature tensor $\Ric_g$ and upper bounds on the topological dimension $n$ play a fundamental role in governing various Sobolev-type inequalities on $(M,g)$, see e.g. \cite{ChavelEigenvalues,GallotBourbaki,GallotIsoperimetricInqs,LiYauEigenvalues,YauIsoperimetricConstantsAndSpectralGap} and the references therein. In the weighted-manifold setting, the pertinent information on \emph{generalized} curvature and \emph{generalized} dimension may be incorporated into a single tensor, which was put forth by Bakry and \'Emery \cite{BakryEmery,BakryStFlour} following Lichnerowicz \cite{Lichnerowicz1970GenRicciTensorCRAS,Lichnerowicz1970GenRicciTensor}. The $N$-dimensional Bakry--\'Emery Curvature tensor ($N \in (-\infty,\infty]$) is defined as (setting $\Psi = \exp(-V)$):
\begin{equation} \label{eq:Ric-def}
\Ric_{\mu,N} := \Ric_g + \nabla^2 V - \frac{1}{N-n} d V\otimes d V = \Ric_g - (N-n) \frac{\nabla^2  \Psi^{\frac{1}{N-n}}}{\Psi^{\frac{1}{N-n}}} ~, 
\end{equation}
and the Bakry--\'Emery Curvature-Dimension condition $\CD(\rho,N)$, $\rho \in \Real$, is the requirement that as 2-tensors on $M$:
\[
\Ric_{\mu,N} \geq \rho g ~.
\]
Here $\nabla^2 V$ denotes the Riemannian Hessian of $V$. Note that the case $N=n$ is only defined when $V$ is constant, i.e. in the classical non-weighted Riemannian setting where $\mu$ is proportional to $\vol_M$, in which case $\Ric_{\mu,n}$ boils down to $\Ric_g$. When $N= \infty$ we set:
\[
\Ric_\mu := \Ric_{\mu,\infty} = \Ric_g + \nabla^2 V ~.
\]

It is customary to only treat the case when $N \in [n,\infty]$, with the interpretation that $N$ is an upper bound on the ``generalized dimension" of the weighted-manifold $(M,g,\mu)$; however, our method also applies with no extra effort to the case when $N \in (-\infty,0]$, and so our results are treated in this greater generality, which in the Euclidean setting encompasses the entire class of Borell's convex (or ``$1/N$-concave") measures \cite{BorellConvexMeasures} (cf. \cite{BrascampLiebPLandLambda1,BobkovLedouxWeightedPoincareForHeavyTails}). It will be apparent that the more natural parameter is actually $1/N$, with $N=\infty,0$ interpreted as $1/N = 0,-\infty$, respectively, and so our results hold in the range $1/N \in [-\infty,1/n]$. As $d V\otimes d V$ appearing in (\ref{eq:Ric-def}) is a positive semi-definite tensor, the $\CD(\rho,N)$ condition is clearly monotone in $\frac{1}{N-n}$ and hence in $\frac{1}{N}$ in the latter range, so for all  $N_+ \in [n,\infty], N_- \in (-\infty,0]$:
\[
\CD(\rho,n) \Rightarrow \CD(\rho,N_+) \Rightarrow \CD(\rho,\infty) \Rightarrow \CD(\rho,N_-) \Rightarrow \CD(\rho,0) ~;
\]
note that $\CD(\rho,0)$ is the weakest condition in this hierarchy. 
 It seems that outside the Euclidean setting, this extension of the Curvature-Dimension condition to the range $N \leq 0$ has not attracted much attention in the weighted-Riemannian and more general metric-measure space setting (cf. \cite{SturmCD12,LottVillaniGeneralizedRicci}); an exception is the work of Ohta and Takatsu \cite{OhtaTakatsuEntropies1,OhtaTakatsuEntropies2}. We expect this gap in the literature to be quickly filled (in fact, concurrently to posting our work on the arXiv, Ohta \cite{Ohta-NegativeN} has posted a first attempt of a systematic treatise of the range $N \leq 0$, and subsequently other authors have also begun treating this extended range \cite{EMilmanNegativeDimension,KlartagLocalizationOnManifolds,Wylie-SectionalCurvature,KennardWylie-WeightedSectionalCurvature,EMilmanHarmonicMeasures}). 

A convenient equivalent form of the $\CD(\rho,N)$ condition may be formulated as follows. Let $\Gamma_2$ denote the iterated carr\'e-du-champ operator of Bakry--\'Emery:
\[
 \Gamma_2(u) := \norm{\nabla^2 u}^2 + \scalar{\Ric_\mu \; \nabla u,\nabla u} ~,
 \]
 where  $\norm{\nabla^2 u}$ denotes the Hilbert-Schmidt norm of $\nabla^2 u$. Then the $\CD(\rho,N)$ condition is equivalent when $1/N \in (-\infty,1/n]$ (see \cite[Section 6]{BakryStFlour} for the case $N \in [n,\infty]$ or Lemma \ref{lem:CS} in the general case) to the requirement that:
\begin{equation}
\label{cdrhon}
\Gamma_2(u) \geq \rho \abs{\nabla u}^2 + \frac{1}{N} (L u)^2 \;\;\; \forall u \in C^2(M) ~.
\end{equation}

Denote by $\S_0(M)$ the class of functions $u$ on $M$ which are $C^2$ smooth in the interior of $M$ and $C^1$ smooth on the entire compact $M$. 
Denote by $\S_N(M)$ the subclass of functions which in addition satisfy that $u_\nu$ is $C^1$ smooth on $\partial M$. 
The main tool we employ in this work is the following: 
\begin{thm}[Generalized Reilly Formula] \label{thm:Reilly}
For any function $u \in \S_N(M)$: 
\begin{multline}
\label{Reilly}
\int_M (L u)^2 d\mu = \int_M \norm{\nabla^2 u}^2 d\mu + \int_M \scalar{ \Ric_\mu \; \nabla u, \nabla u} d\mu + \\
\int_{\partial M} H_\mu (u_\nu)^2 d\mu_{\partial M} + \int_{\partial M} \scalar{\II_{\partial M}  \;\nabla_{\partial M} u,\nabla_{\partial M} u} d\mu_{\partial M} - 2 \int_{\partial M} \scalar{\nabla_{\partial M} u_\nu, \nabla_{\partial M} u} d\mu_{\partial M} ~.
\end{multline}
Here $\nabla_{\partial M}$ denotes the Levi-Civita connection on $\partial M$ with its induced Riemannian metric.
\end{thm}

This natural generalization of the (integrated) Bochner--Lichnerowicz--Weitzenb\"{o}ck formula for manifolds with boundary was first obtained by R.C.~Reilly \cite{ReillyOriginalFormula} in the classical Riemannian setting ($\mu=\vol_M$). The version above is a modification (obtained by integrating by parts on $\partial M$) of a previous version due to L.~Ma and S.-H.~Du \cite{MaDuGeneralizedReilly}. For completeness, we sketch in Section \ref{sec:prelim} the proof of the version (\ref{Reilly}) which we require for deriving our results. 

\medskip

It is known that by dualizing the Bochner--Lichnerowicz--Weitzenb\"{o}ck formula, various Poincar\'e-type inequalities such as the Lichnerowicz \cite{LichnerowiczBook}, Brascamp--Lieb \cite{BrascampLiebPLandLambda1,LedouxSpinSystemsRevisited} and Veysseire \cite{VeysseireSpectralGapEstimateCRAS} inequalities, may be obtained under appropriate bounds on curvature and dimension. Recently, 
heavy-tailed versions of the Brascamp--Lieb inequalities have been obtained in the Euclidean setting by Bobkov--Ledoux \cite{BobkovLedouxWeightedPoincareForHeavyTails} and
sharpened by Nguyen \cite{NguyenDimensionalBrascampLieb}. By employing the generalized Reilly formula, we unify, extend and generalize
many
of these previously known results to various new combinations of boundary conditions on the domain (locally convex, mean-convex) and the function (Neumann, Dirichlet) in the weighted-Riemannian setting. We mention in passing another celebrated application of the latter duality argument in the Complex setting, namely H\"{o}rmander's $L^2$ estimate \cite{Hormander1965L2EstimatesAndDBarProblem}, but we refrain from attempting to generalize it here; further more recent applications may be found in \cite{Helffer-DecayOfCorrelationsViaWittenLaplacian,LedouxSpinSystemsRevisited, KlartagUnconditionalVariance,BartheCorderoVariance,KlartagMomentMap}.

Given a finite measure $\nu$ on a measurable space $\Omega$, and a $\nu$-integrable function $f$ on $\Omega$, we denote:
\[
\dashint_\Omega f d\nu := \frac{1}{\nu(\Omega)} \int_\Omega f d\nu ~,~ \Var_\nu(f) := \int_{\Omega} \brac{f - \dashint_\Omega f d\nu}^2 d\nu ~. 
\]
The following theorem, obtained in Section \ref{sec:BLN}, is the main result of this work: 
\begin{thm}[Generalized Dimensional Brascamp--Lieb With Boundary] \label{thm:intro-BLN}
Assume that $\Ric_{\mu,N} > 0$ on $M$ with $1/N \in (-\infty,1/n]$. Then for any $f \in C^{1}(M)$:
\begin{enumerate}
\item (Neumann Dimensional Brascamp--Lieb inequality on locally convex domain)

Assume that $\II_{\partial M}\geq 0$ ($M$ is locally convex). Then:
 \[
 \frac{N}{N-1} \Var_\mu(f) \leq \int_M \scalar{\Ric_{\mu,N}^{-1} \nabla f, \nabla f} d\mu ~.
 \]

\item (Dirichlet Dimensional Brascamp--Lieb inequality on generalized mean-convex domain)

Assume that $H_\mu \geq 0$ ($M$ is generalized mean-convex), $f \equiv 0$  on $\partial M \neq \emptyset$.
Then:
 \[
  \frac{N}{N-1} \int_M f^2 d\mu\leq \int_M \scalar{\Ric_{\mu,N}^{-1} \nabla f, \nabla f} d\mu ~.
 \]

\item (Neumann Dimensional Brascamp--Lieb inequality on strictly generalized mean-convex domain)
 
Assume that $H_\mu > 0$ ($M$ is strictly generalized mean-convex). Then for any $C \in \Real$:
 \[
 \frac{N}{N-1} \Var_\mu(f) \leq \int_M \scalar{\Ric_{\mu,N}^{-1} \nabla f, \nabla f} d\mu + 
\int_{\partial M} \frac{1}{H_\mu} \Bigl(f - C\Bigr)^2 d\mu_{\partial M} ~.
  \]
  In other words:
\[
 \frac{N}{N-1} \Var_\mu(f) \leq \int_M \scalar{\Ric_{\mu,N}^{-1} \nabla f, \nabla f} d\mu + \Var_{\mu_{\partial M} / H_\mu}(f|_{\partial M})  ~.
\]   
  \end{enumerate}
\end{thm}

Restricting to Euclidean space $(\Real^n,\abs{\cdot})$ and setting $N=\infty$ in Case (1), the tensor $\Ric_{\mu,\infty}$ boils down to the (Euclidean) Hessian $\nabla^2 V$, and we recover 
the celebrated Poincar\'e-type inequality obtained by H. J. Brascamp and E. H. Lieb \cite{BrascampLiebPLandLambda1} as an infinitesimal version of the Prekop\'a--Leindler inequality. 
When $\Ric_{\mu,N} \geq \rho g$ with $\rho > 0$ (i.e. $(M,g,\mu)$ satisfies the $\CD(\rho,N)$ condition), by replacing the $\int_M \sscalar{\Ric_{\mu,N}^{-1} \; \nabla f,\nabla f } d\mu$ term with the \emph{looser} $\frac{1}{\rho} \int_M \abs{\nabla f}^2 d\mu$ in all occurrences above, we obtain 
various generalizations of the classical Lichnerowicz estimate \cite{LichnerowiczBook} on the spectral-gap of the weighted-Laplacian $-L$ under different boundary conditions; in particular, in the non-weighted classical case $N=n$, this recovers the spectral-gap estimate of Escobar \cite{EscobarLichnerowiczWithConvexBoundary} and Xia \cite{XiaLichnerowiczWithConvexBoundary} under Neumann boundary conditions, and the one by Reilly \cite{ReillyOriginalFormula} under Dirichlet conditions. When $N \leq -1$, Case (1) was obtained in the \emph{Euclidean setting} (and under the stronger assumption that $\Ric_{\mu,\infty} = \nabla^2 V > 0$) with a 
constant better
than $\frac{N}{N-1}$ on the left-hand-side above by V. H. Nguyen \cite{NguyenDimensionalBrascampLieb}, improving a previous estimate of S. Bobkov and M. Ledoux \cite{BobkovLedouxWeightedPoincareForHeavyTails} valid when $N \leq 0$. 
However, on a general \emph{weighted Riemannian manifold}, our constant $\frac{N}{N-1}$ is \emph{best possible} in Case (1) for the entire range $N \in (-\infty,-1] \cup [n,\infty]$, see Subsection \ref{subsec:sharp-constant}.

We refer to Subsection \ref{subsec:prev-known} for a longer exposition on the previously known generalizations in these directions; with few exceptions, Cases (2) and (3) and also Case (1) when $N \neq \infty$ seem new. We note that while the heat semi-group approach of Bakry--\'Emery is a very powerful tool in Case (1), namely under Neumann convex boundary conditions, we are not aware of an analogous semi-group approach under the Case (2) Dirichlet mean-convex boundary conditions, let alone the mixed boundary conditions of Case (3), and thus confine our analysis to the $L^2$-duality approach. 

\medskip

To conclude this work, we extend in Section \ref{sec:Vey} a result of L. Veysseire \cite{VeysseireSpectralGapEstimateCRAS}, who obtained a spectral-gap estimate of $1 / \dashint_M (1/\rho) d\mu$ assuming that $\Ric_\mu \geq \rho g$ for a function $\rho : M \rightarrow \Real_+$ which is not necessarily bounded away from zero, 
to the case of Neumann boundary conditions when $M$ is locally convex. 
\begin{rem} \label{rem:non-compact}
Although all of our results are formulated for compact weighted-manifolds with boundary, the results easily extend to the non-compact case, if the manifold $M$ can be exhausted by compact submanifolds $\set{M_k}$ so that each $(M_k,g|_{M_k},\mu|_{M_k})$ has an appropriate boundary (locally-convex or generalized mean-convex, in accordance with the desired result). In the Dirichlet case, the asserted inequalities then extend to all functions in $C^1_0(M)$ having compact support and vanishing on the boundary $\partial M$. In the Neumann cases, the asserted inequalities extend to all functions $f \in C^1_{loc}(M) \cap L^2(M,\mu)$ when $\mu$ is a finite measure. When such an exhaustion is not available but the manifold is complete, one may alternatively apply a functional-analytic argument to obtain analogous results on non-compact manifolds - more details may be found in \cite[Appendix]{KolesnikovEMilman-RiemannianMetrics}.
\end{rem}

\medskip \noindent
\textbf{Acknowledgements.} We thank Franck Barthe, Bo Berndtsson, Andrea Colesanti, Dario Cordero-Erausquin, Bo'az Klartag, Michel Ledoux, Frank Morgan, Van Hoang Nguyen, Shin-ichi Ohta, Yehuda Pinchover and Steve Zelditch for their comments and interest. We also thank the anonymous referees for carefully reading the paper.

\section{Generalized Reilly Formula and Other Preliminaries} \label{sec:prelim}

\subsection{Notation}

We denote by $int(M)$ the interior of $M$. Given a compact differentiable manifold $\Sigma$ (which is at least $C^k$ smooth), we denote by $C^{k}(\Sigma)$ the space of real-valued functions on $\Sigma$ with continuous (and bounded) derivatives $\brac{\frac{\partial}{\partial x}}^a f$, for every multi-index $a$ of order $|a| \leq k$ in a given coordinate system. Similarly, the space $C^{k,\alpha}(\Sigma)$ denotes the subspace of functions whose $k$-th order derivatives are uniformly H\"{o}lder continuous of order $\alpha$ on the $C^{k,\alpha}$ smooth manifold $\Sigma$. When $\Sigma$ is non-compact, we may use $C_{loc}^{k,\alpha}(\Sigma)$ to denote the class of functions $u$ on $M$ so that $u|_{\Sigma_0} \in C^{k,\alpha}(\Sigma_0)$ for all compact subsets $\Sigma_0 \subset \Sigma$. These spaces are equipped with their usual corresponding topologies. 

Throughout this work we employ Einstein summation convention.
By abuse of notation, we denote different covariant and contravariant versions of a tensor in the same manner. So for instance, $Ric_\mu$ may denote the $2$-covariant tensor $(Ric_\mu)_{\alpha,\beta}$, but also may denote its $1$-covariant $1$-contravariant version $(Ric_\mu)^{\alpha}_{\beta}$, as in:
\[
 \scalar{\Ric_\mu \nabla f , \nabla f} = g_{i,j} (\Ric_{\mu})^i_k \nabla^k f \nabla^j f  = (\Ric_\mu)_{i,j} \nabla^i f \nabla^j f = \Ric_\mu(\nabla f,\nabla f) ~.
 \]
 Similarly, inverse tensors are interpreted according to the appropriate context. For instance, the $2$-contravariant tensor $(\text{II}^{-1})^{\alpha,\beta}$ is defined by:
 \[
 (\text{II}^{-1})^{i,j} \text{II}_{j,k} = \delta^i_k ~.
 \]
We freely raise and lower indices by contracting with the metric. 
Since we only deal with $2$-tensors, the only possible contraction is often denoted by using the trace notation $tr$.

Finally, when studying consequences of the $\CD(\rho,N)$ condition, the various expressions in which $N$ appears are interpreted in the limiting sense when $1/N= 0$. For instance, $N/(N-1)$ is interpreted as $1$, and $N f^{1/N}$ is interpreted as $\log f$ (since $\lim_{1/N \rightarrow 0} N (x^{1/N} -1) = \log(x)$; the constant $-1$ in the latter limit does not influence our application of this convention). 
 
 \subsection{Proof of the Generalized Reilly Formula}

For completeness, we sketch the proof of our main tool, Theorem \ref{thm:Reilly} from the Introduction, following the proof given in \cite{MaDuGeneralizedReilly}. 

\begin{proof}[Proof of Theorem \ref{thm:Reilly}]
The generalized Bochner--Lichnerowicz--Weitzenb\"{o}ck formula \cite{Lichnerowicz1970GenRicciTensorCRAS,BakryEmery} 
states that for any $u \in C^3_{loc}(int(M))$, we have:
\begin{equation} \label{eq:Bochner}
\frac{1}{2} L \abs{\nabla u}^2 = \norm{\nabla^2 u}^2 + \scalar{\nabla L u, \nabla u} + \scalar{\Ric_\mu \; \nabla u , \nabla u} ~.
\end{equation}
We introduce an orthonormal frame of vector fields $e_1,\ldots,e_n$ so that $e_n = \nu$ on $\partial M$, and denote $u_i = du(e_i)$, $u_{i,j} = \nabla^2 u(e_i,e_j)$. 
Assuming in addition that $u \in C^2(M)$, we may integrate by parts:
\[
\int_M \frac{1}{2} L \abs{\nabla u}^2 d\mu = \int_{\partial M} \sum_{i=1}^n u_i u_{i,n} d\mu_{\partial M} ~,~ \int_M \scalar{\nabla L u, \nabla u} d\mu = \int_{\partial M} u_n (L u) d\mu_{\partial M} - \int_{M} (L u)^2 d\mu ~. 
\]
Consequently, integrating (\ref{eq:Bochner}) over $M$, we obtain:
\[
\int_M \brac{(L u)^2 - \norm{\nabla^2 u}^2 - \scalar{\Ric_\mu \nabla u , \nabla u}} d\mu = \int_{\partial M} \brac{u_n (Lu) - \sum_{i=1}^n u_i u_{i,n}} d\mu_{\partial M} ~.
\]
Now:
\[
u_n (Lu) - \sum_{i=1}^n u_i u_{i,n}  = \sum_{i=1}^{n-1} \brac{u_n u_{i,i} - u_i u_{i,n}} - u_n \scalar{\nabla u,\nabla V} ~. 
\]
Computing the different terms: 
\begin{eqnarray*}
\sum_{i=1}^{n-1} u_{i,i} &=& \sum_{i=1}^{n-1} \brac{e_i (e_i u) - (\nabla_{e_i} e_i) u} = \sum_{i=1}^{n-1} \brac{e_i (e_i u) - ((\nabla_{\partial M})_{e_i} e_i) u} + \brac{\sum_{i=1}^{n-1} (\nabla_{\partial M})_{e_i} e_i - \nabla_{e_i} e_i } u \\
&=& \Delta_{\partial M} u + \brac{\sum_{i=1}^{n-1} \II_{i,i}} e_n u =  \Delta_{\partial M} u + tr(\II) u_n ~;
\end{eqnarray*}
\[
\sum_{i=1}^{n-1} u_i u_{i,n} = \sum_{i=1}^{n-1} u_i \brac{e_i (e_n u) - (\nabla_{e_i} e_n) u} = \scalar{\nabla_{\partial M} u,\nabla_{\partial M} u_n} - \scalar{\II \;\nabla_{\partial M} u,\nabla_{\partial M} u} ~.
\]
Putting everything together:
\begin{eqnarray*}
& & \int_M \brac{(L u)^2 - \norm{\nabla^2 u}^2 - \scalar{\Ric_\mu \nabla u , \nabla u}} d\mu = \int_{\partial M} \brac{u_n (\Delta_{\partial M} u - \scalar{\nabla u,\nabla V}) + \text{tr(II)} (u_n)^2} d\mu_{\partial M} \\
&-& \int_{\partial M} \scalar{\nabla_{\partial M} u,\nabla_{\partial M} u_n} d\mu_{\partial M} + \int_{\partial M} \scalar{\II \;\nabla_{\partial M} u,\nabla_{\partial M} u} d\mu_{\partial M} ~.
\end{eqnarray*}
This is the formula obtained in \cite{MaDuGeneralizedReilly} for smooth functions. To conclude the proof, simply note that:
\[
\scalar{\nabla u,\nabla V} = \scalar{\nabla_{\partial M} u, \nabla_{\partial M} V} + u_n V_n ~,~ L_{\partial M} = \Delta_{\partial M} - \scalar{\nabla_{\partial M} V,\nabla_{\partial M}} ~,~ H_\mu = tr(\II) - V_n ~,
\]
and thus:
\[
\int_{\partial M} \brac{u_n (\Delta_{\partial M} u - \scalar{\nabla u,\nabla V}) + tr(\II) (u_n)^2} d\mu_{\partial M} = \int_{\partial M} \brac{u_n L_{\partial M} u + H_\mu u_n^2} d\mu_{\partial M} ~.
\]
Integrating by parts one last time, this time on $\partial M$, we obtain:
\[
\int_{\partial M} u_n L_{\partial M} u \; d\mu_{\partial M} = - \int_{\partial M} \scalar{\nabla_{\partial M} u_n , \nabla_{\partial M} u} d\mu_{\partial M} ~.
\]
Finally, plugging everything back, we obtain the asserted formula for $u$ as above:
\begin{eqnarray*}
& & \int_M \brac{(L u)^2 - \norm{\nabla^2 u}^2 - \scalar{\Ric_\mu \; \nabla u , \nabla u}} d\mu \\
& = &
\int_{\partial M} H_\mu u_n^2 d\mu_{\partial M} - 2 \int_{\partial M} \scalar{\nabla_{\partial M} u_n , \nabla_{\partial M} u} d\mu_{\partial M}  + \int_{\partial M} \scalar{\II \;\nabla_{\partial M} u,\nabla_{\partial M} u} d\mu_{\partial M} ~.
\end{eqnarray*}
To conclude that the assertion in fact holds for $u \in \S_N(M)$, we employ a standard approximation argument using a partition of unity and mollification. Since the metric is assumed at least $C^3$ and $\partial M$ is $C^2$, we may approximate any $u \in \S_N(M)$ by functions $u_k \in C^3_{loc}(int(M)) \cap C^2(M)$, so that $u_k \rightarrow u$ in $C^2_{loc}(int(M))$ and $C^1(M)$, and $(u_k)_\nu \rightarrow u_\nu$ in $C^1(\partial M)$. The assertion then follows by passing to the limit.  
\end{proof}

\begin{rem}
For minor technical reasons, it will be useful to record the following variants of the generalized Reilly formula, which are obtained by analogous approximation arguments to the one given above:
\begin{itemize}
\item
If $u_\nu$ or $u$ are constant on $\partial M$ and $u \in \S_0(M)$ (recall $\S_0(M) := C^2_{loc}(int(M)) \cap C^1(M)$), then:
\begin{multline}
\label{Reilly3}
\int_M (L u)^2 d\mu = \int_M \norm{\nabla^2 u}^2 d\mu + \int_M \scalar{ \Ric_\mu \; \nabla u, \nabla u} d\mu + \\
\int_{\partial M} H_\mu (u_\nu)^2 d\mu_{\partial M} + \int_{\partial M} \scalar{\II_{\partial M}  \;\nabla_{\partial M} u,\nabla_{\partial M} u} d\mu_{\partial M} ~.
\end{multline}
\item
If $u \in \S_D(M) := \S_0(M) \cap C^2(\partial M)$, then integration by parts yields:
\begin{multline}
\label{Reilly2}
\int_M (L u)^2 d\mu = \int_M \norm{\nabla^2 u}^2 d\mu + \int_M \scalar{ \Ric_\mu \; \nabla u, \nabla u} d\mu + \\
\int_{\partial M} H_\mu (u_\nu)^2 d\mu_{\partial M} + \int_{\partial M} \scalar{\II_{\partial M}  \;\nabla_{\partial M} u,\nabla_{\partial M} u} d\mu_{\partial M} + 2 \int_{\partial M} u_\nu L_{\partial M} u \; d\mu_{\partial M} ~.
\end{multline}
\end{itemize}
\end{rem}

\begin{rem}
Throughout this work, when integrating by parts, we employ a slightly more general version of the textbook Stokes Theorem $\int_M d\omega = \int_{\partial M} \omega$,  in which one only assumes that $\omega$ is a continuous differential $(n-1)$-form on $M$ which is differentiable on $int(M)$ (and so that $d\omega$ is integrable there); a justification may be found in \cite{MacdonaldGeneralizedStokes}. This permits us to work with the classes $C^k_{loc}(int(M))$ occurring throughout this work. 
\end{rem}

\subsection{The $\CD(\rho,N)$ condition for $1/N \in [-\infty,1/n]$}

The results in this subsection for $1/N \in [0,1/n]$ are due to Bakry (e.g. \cite[Section 6]{BakryStFlour}). 

\begin{lem} \label{lem:CS}
For any $u \in C^2_{loc}(M)$ and $1/N \in [-\infty,1/n]$:
\begin{equation} \label{eq:Bakry-CS}
\Gamma_2(u) = \scalar{ \Ric_\mu \; \nabla u, \nabla u} + \norm{\nabla^2 u}^2 \geq \scalar{\Ric_{\mu,N}\; \nabla u, \nabla u} + \frac{1}{N} (Lu)^2 ~.
\end{equation}
Our convention throughout this work is that $-\infty \cdot 0 = 0$, and so if $Lu = 0$ at a point $p \in M$, the assertion when $\frac{1}{N} = -\infty$  is that:
\[
\Gamma_2(u) \geq \scalar{\Ric_{\mu,0}\; \nabla u, \nabla u} ~,
\]
at that point. 
\end{lem}
\begin{proof}
Recalling the definitions, this is equivalent to showing that:
\[
\norm{\nabla^2 u}^2  + \frac{1}{N-n} \scalar{\nabla u,\nabla V}^2 \geq \frac{1}{N} (Lu)^2 ~.
\]
Clearly the case that $1/N = 0$ ($N = \infty$) follows. But by Cauchy--Schwarz:
\[
\norm{\nabla^2 u}^2  \geq \frac{1}{n} (\Delta u)^2 ~,
\]
and so the case $N=n$, which corresponds to a constant function $V$ so that $\Ric_\mu = \Ric_{\mu,n} = \Ric_g$ and $L = \Delta$, also follows. It remains to show that:
\[
\frac{1}{n} (\Delta u)^2 + \frac{1}{N-n} \scalar{\nabla u,\nabla V}^2 \geq \frac{1}{N} (Lu)^2 ~.
\]
The case $1/N = -\infty$ ($N=0$) follows since when $0 = L u = \Delta u - \scalar{\nabla u,\nabla V}$ then:
\[
\frac{1}{n} (\Delta u)^2 -  \frac{1}{n} \scalar{\nabla u,\nabla V}^2  = \frac{1}{n} (\Delta u + \scalar{\nabla u,\nabla V}) (\Delta u - \scalar{\nabla u,\nabla V}) = 0 ~.
\]
In all other cases, the assertion follows from another application of Cauchy--Schwarz:
\[
\frac{1}{\alpha} A^2 + \frac{1}{\beta} B^2 \geq \frac{1}{\alpha + \beta} (A+B)^2 \;\;\; \forall A,B \in \Real ~,
\]
valid as soon as $(\alpha,\beta)$ lay in either the set $\set{\alpha, \beta > 0}$ or the set $\set{ \alpha + \beta < 0 \text{ and }\alpha \beta < 0}$.
\end{proof}

\begin{rem}
It is immediate to deduce from Lemma \ref{lem:CS} that for $1/N \in (-\infty,1/n]$, $\Ric_{\mu,N} \geq \rho g$ on $M$, $\rho \in \Real$, if and only if:
\[
\Gamma_2(u) \geq \rho \abs{\nabla u}^2 + \frac{1}{N} (Lu)^2 ~,~ \forall u \in C^2_{loc}(M) ~.
\]
Indeed, the necessity follows from Lemma \ref{lem:CS}. The sufficiency follows by locally constructing given $p \in M$ and $X \in T_p M$ a function $u$ so that $\nabla u = X$ at $p$ and equality holds in both applications of the Cauchy--Schwarz inequality in the proof above, as this implies that $\Ric_{\mu,N}(X,X) \geq \rho \abs{X}^2$. Indeed, equality in the first application implies that $\nabla^2 u$ is a multiple of $g$ at $p$, whereas the equality in the second implies when $1/N \notin \set{0,1/n}$ that $\scalar{\nabla u,\nabla V}$ and $\Delta u$ are appropriately proportional at $p$; clearly all three requirements can be simultaneously met. The cases $1/N \in \set{0,1/n}$ follow by approximation. \end{rem}

\subsection{Solution to Poisson Equation on Weighted Riemannian Manifolds}

As our manifold is smooth, connected, compact, with $C^2$ smooth boundary and strictly positive $C^2$-density all the way up to the boundary, all of the classical 
elliptic existence, uniqueness and regularity results (e.g. \cite[Chapter 8]{GilbargTrudinger}, \cite[Chapter 5]{LiebermanObliqueBook}, \cite[Chapter 3]{LadyEllipticBook}) 
immediately extend from the Euclidean setting to our weighted-manifold one (see e.g. \cite{Taylor-PDEBook-I,MorreyBook}); for more general situations (weaker regularity of metric, Lipschitz domains, etc.) see e.g. \cite{MitreaTaylor-PDEonLipManifolds} and the references therein. 
We summarize the results we require in the following:

\begin{thm}
Given a weighted-manifold $(M,g,\mu)$ , $\mu = \exp(-V) d\vol_M$, we assume that $\partial M$ is $C^2$ smooth. Let $\alpha \in (0,1)$, and assume that $g$ is $C^{2,\alpha}$ smooth and $V \in C^{1,\alpha}(M)$. 
Let $f \in C^{0,\alpha}(M)$, $\varphi_D \in C^{2}(\partial M)$ and $\varphi_N \in C^{1}(\partial M)$. Then there exists a function $u \in C^{2,\alpha}_{loc}(int(M)) \cap C^{1,\beta}(M)$ for all $\beta \in (0,1)$, which solves: \[
L u = f ~ \text{on $M$} ~,
\]
with either of the following boundary conditions on $\partial M$:
\begin{enumerate}
\item Dirichlet: $u|_{\partial M} = \varphi_D$, assuming $\partial M \neq \emptyset$.  
\item Neumann: $u_\nu|_{\partial M} = \varphi_N$, assuming the following compatibility condition is satisfied:
\[
 \int_{M} f d\mu = \int_{\partial M} \varphi_N d\mu_{\partial M} ~.
\]
\end{enumerate}
In particular, $u \in \S_0(M)$ in either case. Moreover, $u \in \S_N(M)$ in the Neumann case and $u \in \S_D(M)$ in the Dirichlet case. 

\end{thm}

\begin{rem}
For future reference, we remark that it is enough to only assume in the proof of the generalized Reilly formula (including the final approximation argument) that the metric $g$ is $C^3$ smooth, so in particular the above regularity results apply. 
\end{rem}

We will not require the uniqueness of $u$ above, but for completeness we mention that this is indeed the case for Dirichlet boundary conditions, and up to an additive constant in the Neumann case. 

\subsection{Spectral-gap on Weighted Riemannian Manifolds}

Let $\lambda_1^N$ denote the best constant in the Neumann Poincar\'e inequality:
\[
\lambda_1^N \Var_\mu(f) \leq \int_M \abs{\nabla f}^2 d\mu ~,~ \forall f \in H^1(\mu) ~,
\]
and let $\lambda_1^D$ denote the best constant in the Dirichlet Poincar\'e inequality:
\[
\lambda_1^D \int_M f^2 d\mu \leq \int_M \abs{\nabla f}^2 d\mu ~,~ \forall f \in H^1_0(\mu)   ~.
\]
Here $H^1(\mu)$ and $H^1_0(\mu)$ denote the Sobolev spaces obtained by completing $C^\infty(M)$ and $C^\infty_0(M)$ in the $H^1(\mu)$-norm $\sqrt{\int_M f^2 d\mu + \int_M \abs{\nabla f}^2 d\mu}$.  It is well-known (e.g. \cite{Taylor-PDE-II-Book}) that the symmetric operator $-L$ on $L^2(\mu)$ with domain $C^\infty(M)$ or $C^\infty_0(M)$ admits a (unique) self-adjoint positive semi-definite extension, called the Neumann and Dirichlet (negative) Laplacian, respectively. Both instances have discrete non-negative spectra with corresponding complete orthonormal bases of eigenfunctions. In the first case, $\lambda_1^N$ is the first positive eigenvalue of the (negative) Neumann Laplacian:
\[
-L u = \lambda^N_1 u \text{ on $M$ } ~,~ u_\nu \equiv 0 \text{ on $\partial M$} ~;
\]
the zero eigenvalue corresponds to the eigenspace of constant functions, and so only functions $u$ orthogonal to constants are considered. 
In the second case, $\lambda_1^D$ is the first (positive) eigenvalue of the (negative) Dirichlet Laplacian:\[
-L u = \lambda^D_1 u \text{ on $M$ } ~,~ u \equiv 0 \text{ on $\partial M$} ~.
\]
Our assumptions on the smoothness of $M$, its boundary, and the density $\exp(-V)$, guarantee by elliptic regularity theory that in either case, all eigenfunctions are in $\S_0(M)$ (in fact, in $\S_N(M)$ in the Neumann case and in $\S_D(M)$ in the Dirichlet case).

\section{Generalized Brascamp--Lieb type inequalities on $M$} \label{sec:BLN}

In this section we provide a proof of Theorem \ref{thm:intro-BLN} from the Introduction, which we repeat here for convenience:

\begin{thm}[Generalized Dimensional Brascamp--Lieb With Boundary] \label{thm:BLN}
Assume that $\Ric_{\mu,N} > 0$ on $M$ with $1/N \in (-\infty,1/n]$. Then 
for any  $f \in C^{1}(M)$:
\begin{enumerate}
\item (Neumann Dimensional Brascamp--Lieb inequality on locally convex domain)

Assume that $\II_{\partial M}\geq 0$ ($M$ is locally convex). Then:
 \[
 \frac{N}{N-1} \Var_\mu(f) \leq \int_M \scalar{\Ric_{\mu,N}^{-1} \nabla f, \nabla f} d\mu ~.
 \]

\item (Dirichlet Dimensional Brascamp--Lieb inequality on generalized mean-convex domain)

Assume that $H_\mu \geq 0$ ($M$ is generalized mean-convex), $f \equiv 0$  on $\partial M \neq \emptyset$.
Then:
 \[
  \frac{N}{N-1} \int_M f^2 d\mu\leq \int_M \scalar{\Ric_{\mu,N}^{-1} \nabla f, \nabla f} d\mu ~.
 \]

\item (Neumann Dimensional Brascamp--Lieb inequality on strictly generalized mean-convex domain)
 
Assume that $H_\mu > 0$ ($M$ is strictly generalized mean-convex). Then for any $C \in \Real$:
 \[
 \frac{N}{N-1} \Var_\mu(f) \leq \int_M \scalar{\Ric_{\mu,N}^{-1} \nabla f, \nabla f} d\mu + 
\int_{\partial M} \frac{1}{H_\mu} \Bigl(f - C\Bigr)^2 d\mu_{\partial M} ~.
 \]
 In other words:
  \[
 \frac{N}{N-1} \Var_\mu(f) \leq \int_M \scalar{\Ric_{\mu,N}^{-1} \nabla f, \nabla f} d\mu + \Var_{\mu_{\partial M} / H_\mu}(f|_{\partial M})  ~.
\]   
  \end{enumerate}
\end{thm}

\subsection{Previously Known Particular Cases} \label{subsec:prev-known}

\subsubsection{$1/N=0$ - Generalized Brascamp--Lieb Inequalities}

Recall that when $1/N = 0$, $\Ric_{\mu,N} = \Ric_\mu$, and $\frac{N}{N-1} = 1$.  
When $(M,g)$ is Euclidean space $\Real^n$ and $\mu = \exp(-V) dx$ is a finite measure, the Brascamp--Lieb inequality \cite{BrascampLiebPLandLambda1} asserts that:
\[
\Var_{\mu}(f) \leq \int_{\Real^n}\scalar{(\nabla^2 V)^{-1} \; \nabla f , \nabla f} d\mu ~,~ \forall f \in C^1(\Real^n) ~.
\]
Observe that in this case, $\Ric_{\mu} = \nabla^2 V$, and so taking into account Remark \ref{rem:non-compact}, we see that the Brascamp--Lieb inequality follows from Case (1). The latter is easily seen to be sharp, as witnessed by testing the Gaussian measure in Euclidean space. 

The extension to the weighted-Riemannian setting for $1/N = 0$, at least when $(M,g)$ has no boundary, is well-known to experts, although we do not know who to accredit this to (see e.g. the Witten Laplacian method of Helffer--Sj\"{o}strand \cite{Helffer-DecayOfCorrelationsViaWittenLaplacian} as exposed by Ledoux \cite{LedouxSpinSystemsRevisited}). The case of a locally-convex boundary with Neumann boundary conditions (Case 1 above) can easily be justified in Euclidean space by a standard approximation argument, but this is less clear in the Riemannian setting; probably this can be achieved by employing the Bakry--\'Emery semi-group formalism (see Qian \cite{QianGradientEstimateWithBoundary} and Wang \cite{Wang2010SemiGroupOnManifoldsWithBoundary,WangYanConvexManifolds}). To the best of our knowledge, the other two Cases (2) and (3) are new even for $1/N = 0$. 

\subsubsection{$\Ric_{\mu,N} \geq \rho g$ with $\rho > 0$ - Generalized Lichnerowicz Inequalities} \label{subsec:Lich}

Assume that $\Ric_{\mu,N} \geq \rho g$ with $\rho > 0$, so that $(M,g,\mu)$ satisfies the $\CD(\rho,N)$ condition. It follows that:
\begin{equation} \label{eq:Lich-silly}
\int_M \scalar{\Ric_{\mu,N}^{-1} \; \nabla f, \nabla f} d\mu \leq \frac{1}{\rho} \int_M \abs{\nabla f}^2 d\mu ~,
\end{equation}
and so we may replace in all three cases of Theorem \ref{thm:BLN} every occurrence of the left-hand term in (\ref{eq:Lich-silly}) by the right-hand one. So for instance, Case (1) implies that: 
 \begin{equation} \label{eq:Lich-Estimate}
 \frac{N}{N-1} \Var_\mu(f) \leq \frac{1}{\rho} \int_M \abs{\nabla f}^2 d\mu ~,
 \end{equation}
and similarly for the other two cases; we refer to the resulting inequalities as Cases (1'), (2') and (3'). Clearly, Cases (1') and (2') are spectral-gap estimates for $-L$ with Neumann and Dirichlet boundary conditions, respectively.
 
Recall that in the non-weighted Riemannian setting ($\mu = \vol_M$ and $N=n$), $\Ric_{\mu,N} = \Ric_g$. In this classical setting, the above spectral-gap estimates are due to the following authors: when $\partial M = \emptyset$, Cases (1') and (3') degenerate to a single statement, due to Lichnerowicz \cite{LichnerowiczBook}, and by Obata's theorem \cite{Obata-EqualityInLichnerowicz} equality is attained if and only if $M$ is the $n$-sphere. When $\partial M \neq \emptyset$, Case (1') is due to Escobar \cite{EscobarLichnerowiczWithConvexBoundary} and independently Xia \cite{XiaLichnerowiczWithConvexBoundary} ; Case (2') is due to Reilly \cite{ReillyOriginalFormula} ; in both cases, one has equality if and only if $M$ is the $n$-hemisphere ; Case (3') seems new even in the classical case. 

On weighted-manifolds with $N \in [n,\infty]$, Case (1') is certainly known, see e.g. \cite{LiWei-SpectralGapEstimatesAndRigidityForWeightedManifolds} (in fact, a stronger log-Sobolev inequality goes back to Bakry and \'Emery \cite{BakryEmery}); Case (2') was recently obtained under a slightly stronger assumption by Ma and Du \cite[Theorem 2]{MaDuGeneralizedReilly}; for an adaptation to the $\CD(\rho,N)$ condition see Li and Wei \cite[Theorem 3]{LiWei-SpectralGapEstimatesAndRigidityForWeightedManifolds}, who also showed that in both cases one has equality if and only if $N=n$ and $M$ is the $n$-sphere or $n$-hemisphere endowed with its Riemannian volume form, corresponding to whether $\partial M$ is empty or non-empty, respectively. As already mentioned,  Case (3') seems new. 

To the best of our knowledge, the case of $N<0$ has not been previously treated in the Riemannian setting. Concurrently to posting our work on the arXiv,  Ohta \cite{Ohta-NegativeN} has also obtained Case (1') for $N < 0$ when $\partial M = \emptyset$.

\subsubsection{Generalized Bobkov--Ledoux--Nguyen Inequalities}

In the \emph{Euclidean setting} with $N \leq -1$ (and under the stronger assumption that $\Ric_\mu = \nabla^2 V > 0$),  Case (1) with a better constant of $\frac{n-N-1}{n-N}$ instead of our $\frac{N}{N-1} = \frac{-N}{-N+1}$ is due to Nguyen \cite[Proposition 10]{NguyenDimensionalBrascampLieb}, who generalized and sharpened a previous version valid for $N \leq 0$ by Bobkov--Ledoux \cite{BobkovLedouxWeightedPoincareForHeavyTails}. 
However, on a general \emph{weighted Riemannian manifold}, our constant $\frac{N}{N-1}$ is best possible in the range $N \in (-\infty,-1] \cup [n,\infty]$, see Subsection \ref{subsec:sharp-constant} below.

Note that in the Euclidean case, the $\CD(0,N)$ condition with $N \in \Real$ corresponds to Borell's class of convex measures \cite{BorellConvexMeasures}, also known as ``$1/N$-concave measures" (cf. \cite{EMilmanRotemHomogeneous}). When $N<0$, these measures are heavy-tailed, having tails decaying to zero only polynomially fast, and consequently the corresponding generator $-L$ may not have a strictly positive spectral-gap. This is compensated by the weight $\Ric_{\mu,N}^{-1}$ in the resulting Poincar\'e-type inequality. A prime example is given by the Cauchy measure in $\Real^n$, which satisfies $\CD(0,0)$ (it is $-\infty$-concave). See \cite{BobkovLedouxWeightedPoincareForHeavyTails,NguyenDimensionalBrascampLieb} for more information. 

Still in the Euclidean setting with $N \geq n$ (in fact $N > n-1$), a dimensional version of the Brascamp--Lieb inequality which is reminiscent of Case (1) was obtained by Nguyen \cite[Theorem 9]{NguyenDimensionalBrascampLieb}. 
The Bobkov--Ledoux results were obtained as an infinitesimal version of the Borell--Brascamp--Lieb inequality \cite{BorellConvexMeasures,BrascampLiebPLandLambda1} - a generalization of the Brunn-Minkowski inequality, which is strictly confined to the Euclidean setting. Nguyen's approach is already more similar to our own, dualizing an ad-hoc Bochner formula obtained for a non-stationary diffusion operator. 

In any case, our unified formulation (and treatment) of both regimes $N \leq 0$ and $N \in [n,\infty]$, the weaker assumption that $\Ric_{\mu,N} > 0$, the extension to the Riemannian setting with sharp constant $\frac{N}{N-1}$ and the treatment of the different boundary conditions in Cases (1), (2) and (3) seem new.

\subsection{Sharpness of the $\frac{N}{N-1}$ constant in the Riemannian setting} \label{subsec:sharp-constant}

We briefly comment on the sharpness of the constant $\frac{N}{N-1}$ for the range $N \in (-\infty,-1] \cup [n,\infty]$ in the more traditional setting of Case (1); the sharpness of Case (2) is also shown for $N \geq n$. This constant is no longer sharp in Case (1) for $N < 0$ with $\abs{N} \ll 1$, since under the $\CD(\rho,N)$ condition with $\rho > 0$, the spectral-gap remains bounded below as $N<0$ increases to $0$, see \cite{EMilmanNegativeDimension}. 

As described in Subsection \ref{subsec:Lich}, it is classical that equality in the Lichnerowicz estimate (\ref{eq:Lich-Estimate}) is attained by the $n$-sphere and $n$-hemisphere in Cases (1) (and (3)) and by the $n$-hemisphere in Case (2), both endowed with the usual Riemannian volume. This demonstrates the sharpness of  the constant $\frac{N}{N-1}$ when $N=n$.

For general $N \in (-\infty,-1] \cup (n,\infty]$, the sharpness may be shown as follows. Given $\rho > 0$, set $\delta = \frac{\rho}{N-1}$ and:
\[
\beta := \begin{cases} \frac{\pi}{2 \sqrt{\delta}} & \delta > 0 \\ \infty & \delta < 0 \end{cases} \; , \; \alpha := \begin{cases} -\beta & \text{Case (1)} \\ 0 & \text{Case (2)} \end{cases} .
\]
Define the following functions of $t \in [\alpha,\beta]$:
\[
R(t) := \begin{cases}\cos(\sqrt{\delta} t) & \delta > 0 \\ \cosh(\sqrt{-\delta} t) & \delta < 0 \end{cases} ~,~ \Psi_{N-1}(t) := R^{N-1}(t) ~.
\] 
If we extend our setup to include the case of one-dimensional ($n=1$) weighted manifolds, namely the case of the real line endowed with a density, then it is immediate to check that $([\alpha,\beta],\abs{\cdot},\mu = \Psi_{N-1}(t) dt)$ satisfies the $\CD(\rho,N)$ condition, since:
\[
\Ric_{\mu,N} = -(N-1) \frac{(\Psi_{N-1}^{\frac{1}{N-1}})''}{\Psi_{N-1}^{\frac{1}{N-1}}} = -(N-1) \frac{R''}{R} = (N-1) \delta = \rho ~.
\]
Note that when $n=1$, our constant $\frac{N}{N-1}$ and Nguyen's one $\frac{n-N-1}{n-N}$ coincide. As we have learned from Nguyen, his constant is sharp in the Euclidean setting for any $n \geq 1$. One consequently verifies the sharpness for $n=1$ by using the same test function used by Nguyen in \cite{NguyenDimensionalBrascampLieb}, namely $f(t) = \frac{d}{dt} R(t)$. Indeed, when $N < -1$ or $N > 1$ (to ensure convergence of the integrals below) we have:
\[
\int f(t) d\mu = \int_{-\beta}^\beta R'(t) R^{N-1}(t) dt = \frac{1}{N} \int_{-\beta}^\beta (R^N(t))' dt = 0 ~,
\]
since $\lim_{t \rightarrow \beta} R^N(t) = 0$, and since also $f(0) = R'(0) = 0$ (so that the Dirichlet boundary condition at $t=0$ is satisfied in Case (2)), we may integrate by parts:
\[
\int f^2(t) d\mu = \frac{1}{N} \int_{\alpha}^\beta R'(t) (R^N(t))' dt = -\frac{1}{N}\int_{\alpha}^\beta R''(t) R^N(t) dt = \frac{\rho}{N (N-1)} \int_{\alpha}^\beta R^{N+1}(t) dt ~.
\]
On the other hand:
\[
\int \Ric_{\mu,N}^{-1} f'(t)^2 d\mu = \frac{1}{\rho} \int_{\alpha}^\beta (R''(t))^2 R^{N-1}(t) dt = \frac{\rho}{(N-1)^2} \int_{\alpha}^\beta R^{N+1}(t) dt ~.
\]
Comparing the last two expressions, we conclude the sharpness of the constant $\frac{N}{N-1}$ for $n=1$ in Case (1) when $\abs{N} > 1$ and in Case (2) when $N > 1$ (the function $f(t)$ does not vanish at infinity when $N < 0$ so this range is excluded in Case (2)). When $N = -1$, one uses an appropriately truncated version of the above test function. In any case, to assert sharpness for a \emph{compact} weighted manifold with strictly positive density, we truncate the above construction at a finite $\beta_\eps \in (0,\beta)$, and let $\beta_\eps$ tend to $\beta$. 

To see the sharpness for $n \geq 2$, we proceed by repeating the construction from \cite{EMilmanSharpIsopInqsForCDD}, which emulates the above $1$-dimensional model space on a thin weighted $n$-dimensional manifold of revolution. For $n \geq 3$, define:
\[
\Psi_{N-n}(t) := R^{N-n}(t) ,
\]
and given $\eps > 0$, consider the $n$-dimensional manifold $M := [\alpha,\beta] \times S^{n-1}$ endowed with the metric $g_\eps$ and measure $\mu_\eps$ given by:
\begin{align*}
g_\eps & := dt^2 + \eps^2 R(t)^2 g_{S^{n-1}} ~; \\
\mu_\eps & := \Psi(t,\theta) dvol_{g_\eps}(t,\theta) ~,~  \Psi(t,\theta) = \Psi_{N-n}(t) ~,~ (t,\theta) \in [\alpha,\beta] \times S^{n-1} ~.
\end{align*}
The intuition behind this construction is that when $\eps > 0$ is small enough, the geometry of $(M,g_\eps)$ will contribute (at least) $(n-1)\delta$ to the generalized Ricci curvature tensor $\Ric_{g,\mu,N}$, and a factor of $R^{n-1}(t)$ to the density $d\mu_{\eps} \brac{(-\infty,t] \times S^{n-1}}/dt$, whereas the measure $\mu_\eps$ will contribute $(N-n) \delta g_\eps$ to the former and a factor of $\Psi_{N-n}(t) = R^{N-n}(t)$ to the latter, totaling $(N-1) \delta = \rho$ and $R^{N-1}(t) = \Psi_{N-1}(t)$, respectively. Consequently $(M,g_\eps,\mu_\eps)$ satisfies the $\CD(\rho,N)$ condition for small enough $\eps > 0$, and its measure projection onto the axis of revolution is $c_\eps \Psi_{N-1}(t)$; the sharpness of the constant then follows from our previous one-dimensional analysis. Note that in Case (2), the boundary component $\set{0} \times S^{n-1}$ is totally geodesic and hence satisfies our boundary curvature assumptions. In practice, when $N \geq n$ (and thus $\beta < \infty$), we need to ensure that the resulting compact weighted manifold is smooth at its vertices (at $t \in \set{-\beta,\beta}$ in Case (1) and $t = \beta$ in Case (2)), and this is achieved as in \cite{EMilmanSharpIsopInqsForCDD} by gluing appropriate caps. When $N \leq -1$ (and thus $\beta = \infty$), in order to obtain a compact manifold as in the formulation of Theorem \ref{thm:BLN}, we also need to truncate the above construction at a finite $\beta_\eps > 0$; the resulting boundary $\set{-\beta_\eps,\beta_\eps} \times S^{n-1}$ turns out to indeed be locally convex since $R'(\beta_\eps) = - R'(-\beta_\eps) > 0$, according to the calculation in \cite{EMilmanSharpIsopInqsForCDD}. The construction is even more complicated for the case $n=2$; we refer to \cite{EMilmanSharpIsopInqsForCDD} for further precise details and rigorous justifications. 

\subsection{Proof of Theorem \ref{thm:BLN}}

\begin{proof}[Proof of Theorem \ref{thm:BLN}]
Plugging (\ref{eq:Bakry-CS}) into the generalized Reilly formula, we obtain for any $u \in \S_N(M)$:
\begin{multline}
\label{eq:CD-Reilly-BLN}
\frac{N-1}{N} \int_M (L u)^2 d\mu \geq \int_M \scalar{\Ric_{\mu,N} \; \nabla u, \nabla u} d\mu + \\
\int_{\partial M} H_\mu (u_\nu)^2 d\mu_{\partial M} + \int_{\partial M} \scalar{\II_{\partial M}  \;\nabla_{\partial M} u,\nabla_{\partial M} u} d\mu_{\partial M} - 2 \int_{\partial M} \scalar{\nabla_{\partial M} u_\nu, \nabla_{\partial M} u} d\mu_{\partial M} ~.
\end{multline}
Recall that this remains valid for $u \in \S_0(M)$ if $u$ or $u_\nu$ are constant on $\partial M$. 
Lastly, note that if $Lu = f$ in $M$ with $f \in C^{1}(M)$ and $u \in \S_0(M)$, then:
\begin{equation} \label{eq:base-BLN}
  \int_{M} f^2 d\mu = \int_M (Lu)^2 d\mu = \int_M f Lu \; d\mu = -\int_M \scalar{\nabla f ,\nabla u} d\mu + \int_{\partial M} f u_\nu d\mu_{\partial M} ~.
\end{equation}
Consequently, by Cauchy--Schwarz:
\begin{equation} \label{eq:duality-BLN}
\int_M f^2 d\mu \leq \brac{\int_M \scalar{\Ric_{\mu,N} \; \nabla u, \nabla u} d\mu}^{1/2} \brac{ \int_M \scalar{\Ric_{\mu,N}^{-1} \; \nabla f, \nabla f} d\mu}^{1/2} + \int_{\partial M} f u_\nu d\mu_{\partial M} ~.
\end{equation}
We now proceed to treat the individual three cases. 

\begin{enumerate}
\item
Assume that $\int_M f d\mu = 0$ and solve the Neumann Poisson problem  for $u \in \S_0(M)$:
\[
Lu =f \text{ on $M$ } ~,~ u_\nu \equiv 0 \text{ on $\partial M$} ~;
\]
note that the compatibility condition $\int_{\partial M} u_\nu d\mu_{\partial M} = \int_M f d\mu = 0$ is indeed satisfied, so a solution exists.
Since $u_\nu|_{\partial M} \equiv 0$ and $\II_{\partial M}\geq 0$, we obtain from (\ref{eq:CD-Reilly-BLN}):
\begin{equation} \label{eq:CD-BLN}
\frac{N}{N-1}\int_M \scalar{\Ric_{\mu,N} \; \nabla u, \nabla u} d\mu \leq \int_M (Lu)^2 d\mu =  \int_M f^2 d\mu ~.
 \end{equation}
Plugging this back into (\ref{eq:duality-BLN}) and using that $u_\nu \equiv 0$ yields the assertion of Case (1). 
 
\item
Assume that $f|_{\partial M} \equiv 0$ and solve the Dirichlet Poisson problem for $u \in \S_0(M)$:
\[
Lu =f \text{ on $M$ } ~,~ u \equiv 0 \text{ on $\partial M$} ~.
\]
Observe that (\ref{eq:CD-BLN}) still holds since $u|_{\partial M} \equiv 0$ and $H_{\mu} \geq 0$. Plugging (\ref{eq:CD-BLN}) back into  (\ref{eq:duality-BLN}) and using that $f|_{\partial M}\equiv 0$ yields the assertion of Case (2). 

\item
Assume that $\int_M f d\mu = 0$ and solve the Dirichlet Poisson problem:
 \[
Lu =f \text{ on $M$ } ~,~ u \equiv 0 \text{ on $\partial M$} ~.
\]
The difference with the previous case is that the $\int f u_\nu d\mu_{\partial M}$ term in (\ref{eq:base-BLN}) does not vanish since we do not assume that $f|_{\partial M} \equiv 0$. Consequently, we cannot afford to omit the positive contribution of $\int_{\partial M} H_{\mu} (u_\nu)^2 d\mu_{\partial M}$ in (\ref{eq:CD-Reilly-BLN}):
\[
\frac{N-1}{N} \int_M f^2 d\mu \geq \int_M \scalar{\Ric_{\mu,N} \; \nabla u, \nabla u} d\mu + \int_{\partial M} H_\mu u_\nu^2 d\mu_{\partial M} ~.
 \]
 Applying the duality argument, this time in additive form, we obtain for any $\lambda > 0$:
 \begin{eqnarray*}
 \int_M f^2 d\mu & = & -\int_M \scalar{\nabla f,\nabla u} d\mu + \int_{\partial M} f u_\nu d\mu_{\partial M} \\
 &\leq & \frac{1}{2 \lambda} \int_M \scalar{\Ric_{\mu,N}^{-1} \; \nabla f, \nabla f} d\mu + \frac{\lambda}{2} \int_M\scalar{\Ric_{\mu,N} \; \nabla u, \nabla u} d\mu + \int_{\partial M} f u_\nu d\mu_{\partial M} ~.
 \end{eqnarray*}
 Since $\int_{\partial M} u_\nu d\mu_{\partial M} = \int_M f d\mu = 0$, we may as well replace the last term by $\int_{\partial M}  (f-C) u_\nu d\mu_{\partial M}$. 
 Plugging in the previous estimate and applying the Cauchy--Schwarz inequality again to eliminate $u_\nu$, we obtain:
 \begin{align*}
 & \brac{1 - \frac{\lambda}{2} \frac{N-1}{N}} \int_M f^2 d\mu \\
 & \leq \frac{1}{2 \lambda} \int _M\scalar{\Ric_{\mu,N}^{-1} \; \nabla f, \nabla f} d\mu + \int_{\partial M}  (f-C) u_\nu d\mu_{\partial M} - \frac{\lambda}{2} \int_{\partial M} H_\mu u_\nu^2 d\mu_{\partial M} \\
 & \leq \frac{1}{2 \lambda} \int_M \scalar{\Ric_{\mu,N}^{-1} \; \nabla f, \nabla f} d\mu + \frac{1}{2 \lambda} \int_{\partial M} \frac{1}{H_\mu} (f - C)^2  d\mu_{\partial M} ~.
 \end{align*}
 Multiplying by $2 \lambda$ and using the optimal $\lambda = \frac{N}{N-1}$, we obtain the assertion of Case (3).  \end{enumerate}
 \end{proof}

\section{Generalized Veysseire Spectral-gap inequality on convex $M$} \label{sec:Vey}

The next result was recently obtained by L. Veysseire \cite{VeysseireSpectralGapEstimateCRAS} for compact weighted-manifolds without boundary. It may be thought of as a spectral-gap version of the Generalized Brascamp--Lieb inequality. We provide an extension in the case that $M$ is locally convex.

\begin{thm}[Veysseire Spectral-Gap inequality with locally-convex boundary] \label{thm:Veysseire}
Assume that as $2$-tensors on $M$:
\[
\Ric_\mu \geq \rho g ~,
\]
for some measurable function $\rho : M \rightarrow \Real_+$. Then 
for any $f \in C^{1}(M)$:

\begin{enumerate}
\item (Neumann Veysseire inequality on locally convex domain)

Assume that $\II_{\partial M}\geq 0$ ($M$ is locally convex). Then:
 \[
 \Var_\mu(f) \leq \dashint_M \frac{1}{\rho} d\mu \; \int_M \abs{\nabla f}^2 d\mu.
 \]

 \end{enumerate}
\end{thm}

\begin{rem}
We do not know whether the analogous results for Dirichlet or Neumann boundary conditions (Cases (2) and (3) in the previous section) hold on a generalized mean-convex domain, as the proof given below breaks down in those cases. 
\end{rem}

\begin{rem}
As in Veysseire's work \cite{VeysseireSpectralGapEstimateCRAS}, further refinements are possible. For instance, if in addition the $\CD(\rho_0,N)$ condition is satisfied for $\rho_0 > 0$ and $1/N \in [-\infty,1/n]$, then one may obtain an estimate on the corresponding spectral-gap $\lambda_1^N$ of the form:
\[
\lambda_1^N \geq \frac{N}{N-1} \rho_0 + \frac{1}{\dashint_M \frac{1}{\rho - \rho_0} d\mu} ~.
\]
As explained in \cite{VeysseireSpectralGapEstimateCRAS}, this may be obtained by using an appropriate convex combination of the Lichnerowicz estimate (Case (1) of Theorem \ref{thm:intro-BLN} after replacing $Ric_{\mu,N}^{-1}$ with $1/\rho_0$)  and the estimates obtained in this section, with a final application of the Cauchy--Schwarz inequality. 
Similarly, it is possible to interpolate between the Lichnerowicz estimates and the Dimensional Brascamp--Lieb ones of Theorem \ref{thm:intro-BLN}. We leave this to the interested reader. 
\end{rem}

Veysseire's proof in \cite{VeysseireSpectralGapEstimateCRAS} is based on the Bochner formula and the following observation, valid for any $u \in C^2(M)$ at any point so that $\nabla u \neq 0$:
\begin{equation} \label{eq:Vey}
\norm{D^2 u} \geq \abs{\nabla \abs{\nabla u}} ~.
\end{equation}
At a point where $\nabla u = 0$, we define $\abs{\nabla \abs{\nabla u}} := 0$. 

\begin{proof}[Proof of Theorem \ref{thm:Veysseire}]
Plugging (\ref{eq:Vey}) into the generalized Reilly formula and integrating the $\int_M (L u)^2 d\mu$ term by parts, we obtain for any $u \in \S_N(M)$ so that $L u \in C^1(M)$:
\begin{multline} \label{eq:VeyReilly}
\int_{\partial M} u_\nu Lu d\mu - \int_M \scalar{\nabla u,\nabla Lu} d\mu \geq \int_M \abs{\nabla \abs{\nabla u}}^2 d\mu +  \int_M \scalar{ \Ric_\mu \; \nabla u, \nabla u} d\mu + \\
\int_{\partial M} H_\mu (u_\nu)^2 d\mu_{\partial M} + \int_{\partial M} \scalar{\II_{\partial M}  \;\nabla_{\partial M} u,\nabla_{\partial M} u} d\mu_{\partial M} - 2 \int_{\partial M} \scalar{\nabla_{\partial M} u_\nu, \nabla_{\partial M} u} d\mu_{\partial M} ~.
\end{multline}

Let $u \in \S_N(M)$ denote an eigenfunction of $-L$ with zero Neumann boundary conditions corresponding to $\lambda_1^N$, so that in particular $Lu = -\lambda_1^N u \in C^1(M)$, and denote $h = \abs{\nabla u} \in H^1(\mu)$. Applying (\ref{eq:VeyReilly}) to $u$, using that $\II_{\partial M} \geq 0$, and that $\int_{\set{h=0}} \abs{\nabla h}^2 d\mu = 0$ for any $h \in H^1(\mu)$, we obtain:
\[
\lambda_1^N \int_M h^2 d\mu \geq \int_M \abs{\nabla h}^2 d\mu + \int_M \rho h^2 d\mu ~.
\]
Applying the Neumann Poincar\'e inequality to the function $h$, we obtain:
\[
\lambda_1^N \int_M h^2 d\mu \geq \lambda_1^N \brac{ \int_M h^2 d\mu - \frac{1}{\mu(M)} (\int_M h d\mu)^2} + \int_M \rho h^2 d\mu ~.
\]
It follows by Cauchy--Schwarz that:
\[
\lambda_1^N \geq \frac{\mu(M) \int_M \rho h^2 d\mu}{(\int_M h d\mu)^2} \geq \frac{\mu(M)}{\int_M \frac{1}{\rho} d\mu} ~,
\]
concluding the proof.

\end{proof}

\begin{rem}
The proof above actually yields a meaningful estimate on the spectral-gap $\lambda_1^N$ even when $\II_{\partial M}$ is negatively bounded from below. However, this estimate depends on upper bounds on $\abs{\nabla u}$, where $u$ is the first non-trivial Neumann eigenfunction, both in $M$ and on its boundary. 
\end{rem}

\vspace{-20pt}

\bibliographystyle{plain}
\bibliography{../../ConvexBib}

\end{document}